\documentclass[11pt,reqno]{amsart}

\usepackage[T1]{fontenc}
\usepackage{amsmath, amssymb, amsthm, amsfonts}
\usepackage{mathtools}
\usepackage[dvipsnames]{xcolor}

\definecolor{DartmouthGreen}{HTML}{00693E}
\definecolor{DartmouthDarkGreen}{HTML}{004B23}
\definecolor{DartmouthLightGreen}{HTML}{4BAF7A}

\usepackage[
  colorlinks=true,
  hyperindex,
  linkcolor=DartmouthDarkGreen,
  citecolor=DartmouthGreen,
  urlcolor=DartmouthLightGreen,
  pagebackref=true
]{hyperref}

\usepackage{fullpage}
\usepackage{tikz,tikz-cd}
\usetikzlibrary{graphs, graphs.standard,positioning}

\usepackage{colortbl, array,booktabs,tabularx}
\usepackage{enumerate}
\usepackage[shortlabels]{enumitem}

\usepackage{calrsfs}
\DeclareMathAlphabet{\pazocal}{OMS}{zplm}{m}{n}
\usepackage{newcent,fouriernc}

\newtheorem{theorem}{Theorem}[section]
\newtheorem{lemma}[theorem]{Lemma}

\newtheorem{proposition}[theorem]{Proposition}
\newtheorem{conjecture}[theorem]{Conjecture}

\theoremstyle{definition}

\newtheorem{example}[theorem]{Example}

\newtheorem{remark}[theorem]{Remark}

\newtheorem{theoremalpha}{Theorem}
\newtheorem{corollaryalpha}[theoremalpha]{Corollary}

\newcommand{\uu}{\mathbf{u}}
\newcommand{\vv}{\mathbf{v}}
\newcommand{\ww}{\mathbf{w}}
\newcommand{\zz}{\mathbf{z}}
\newcommand{\pp}{\mathbf{p}}
\newcommand{\qq}{\mathbf{q}}
\newcommand{\cc}{\mathbf{c}}
\newcommand{\zero}{\mathbf{0}}

\newcommand{\KK}{\mathbb{K}}

\newcommand{\RR}{\mathbb{R}}
\newcommand{\ZZ}{\mathbb{Z}}

\newcommand{\cB}{\mathcal{B}}

\newcommand{\mfm}{\mathfrak{m}}

\DeclareMathOperator{\gr}{gr}

\DeclareMathOperator{\Sym}{Sym}
\DeclareMathOperator{\vol}{vol}

\numberwithin{equation}{section}

\title{Eventual Nonstandard Koszulness Fails for Veronese Subrings of Weighted Polynomial Rings}

\author{Juliette Bruce}
\address{Department of Mathematics, Dartmouth College, Hanover, NH}
\email{\href{mailto:juliette.bruce@dartmouth.edu}{juliette.bruce@dartmouth.edu}}
\urladdr{\url{https://www.juliettebruce.xyz}}


\begin{document}


\maketitle

\section{Introduction}

An organizing heuristic throughout projective geometry is that defining equations, syzygies, and other algebraic properties of varieties embedded in projective space simplify as the positivity of the embedding line bundle increases. One instance of this heuristic is the Koszul property. If $\KK$ is a field, a standard $\ZZ$-graded $\KK$-algebra $A$ is \emph{Koszul} if the minimal graded $A$-free resolution of the residue field is linear \cite{priddy70}. Backelin proved that if $A$ is a finitely generated standard $\ZZ$-graded $\KK$-algebra, then the $e$-th Veronese subring $A^{(e)} \coloneqq \bigoplus_{k\geq0} A_{ek}$ is Koszul, in the sense of Priddy, for all sufficiently large $e$ \cite{backelin86}.  Eisenbud, Reeves, and Totaro strengthened this, showing that if $\KK$ is infinite, the defining ideal of $A^{(e)}$ has a quadratic Gr\"{o}bner basis for all $e\gg0$ \cite{ERT94}. 

Recent work extends this positivity heuristic to nonstandard and multigraded coordinate rings. Geometrically, these encode subvarieties of weighted projective spaces and more general toric varieties, respectively. Let $A$ be a finitely generated positively $\ZZ$-graded $\KK$-algebra with $A_0=\KK$, and set $\mfm=A_+$. Following Herzog, Reiner, and Welker, we say $A$ is a \emph{nonstandard Koszul algebra} if  and only if the associated graded ring of  $A$ with respect to $\mfm$:
\[
\gr_{\mfm}(A) \coloneqq \bigoplus_{k\geq0} \mfm^{k}/\mfm^{k+1}=\KK\oplus \mfm/\mfm^{2}\oplus \mfm^{2}/\mfm^{3}\oplus \cdots,
\]
is Koszul as a standard $\ZZ$-graded $\KK$-algebra \cite[Definition~5.1]{HRW98}. In particular, $\gr_{\mfm}(A)$ must have a presentation where the ideal of relations is minimally generated by quadrics \cite[Remark~2.6(1)]{CDR13}. Davis, Erman, and Martinova conjecture that an analogue of Backelin's theorem holds for nonstandard $\ZZ$-graded polynomial rings. Let $\ww\in \ZZ^n_{\geq1}$ and $S=\KK[x_1,\ldots,x_n]$ be a $\ZZ$-graded polynomial ring where $\deg(x_i) = w_i$.

\begin{conjecture}\cite[Conjecture~1.5]{DM25}\label{con:DEM}
	If $S$ is a nonstandard graded polynomial ring and $S^{(e)}$ is the $e$-th Veronese subring, then $S^{(e)}$ is a nonstandard Koszul algebra for $e$ sufficiently large. 
\end{conjecture}

In contrast to the standard-graded case the qualifier ``for $e$ sufficiently large'' is necessary. For example, if $S=\KK[x_1,x_2,x_3]$ and $w=(3,4,5)$ then $S^{(15)}$ is not nonstandard Koszul \cite[Remark~1.6]{DM25}. Chase, Fiorindo, Holleben, Marangone, Nguy{\~{\^e}}n, Seceleanu, and Singh produced two parametrized families generalizing this example \cite[Proposition~4.3]{CFHMNSS26}.  For each fixed weight vector, however, their constructions specify only one bad value of $e$ and therefore do not disprove eventual nonstandard Koszulness. Our first result gives a fixed four-variable grading with bad Veronese subrings in an infinite arithmetic progression disproving Conjecture~\ref{con:DEM}.

\begin{theoremalpha}\label{thm:main-ex}
	Let $S=\KK[x_1,x_2,x_3,x_4]$ be the nonstandard $\ZZ$-graded polynomial ring where $\deg(x_1)=1, \deg(x_2)=4, \deg(x_3)=7, \deg(x_4)=9$. If $e \geq 21$ and $e \equiv 3 \pmod{9}$ then $S^{(e)}$ is not nonstandard Koszul. 
\end{theoremalpha}

The conjecture was nevertheless supported by several general results:
\begin{enumerate}
	\item For every nonstandard $\ZZ$-graded polynomial ring $S$ there are infinitely many $e$ for which $S^{(e)}$ is a nonstandard Koszul algebra. If $\ell$ is the least common multiple of $w_1,\ldots,w_n$ consider the lattice polytope $P \coloneqq \left\{ \alpha \in \RR^n_{\geq0} \;\; \big| \;\; \ww\cdot \alpha = \ell \right\}$, which is full dimensional in the hyperplane defined by $\ww\cdot \alpha = \ell$.  Letting $c=\max\{1,n-2\}$ the dilation $cP$ is normal by \cite[Theorem~2.2.12]{CLS11}. Since the monomials in $S_{c\ell}$ are in bijection with the lattice points of $cP$ normality implies $S^{(c\ell)}$ is generated in degree one. So $S^{(c\ell d)}=(S^{(c\ell)})^{(d)}$ is Koszul for  $d\gg1$ by Backelin.  
	\item If $n=1$ or $2$ then $S^{(e)}$ is nonstandard Koszul for all $e\geq1$. The one-variable case is immediate. When $n=2$, $S^{(e)}$ is a two-dimensional normal affine semigroup ring, and so the result follows from \cite[Proposition~5.3]{HRW98}. The recent work of Chase et al. strengthens this by giving an explicit determinantal presentation and Gr\"obner basis \cite[Theorem~3.5]{CFHMNSS26}.
	\item More generally, Herzog, Reiner, and Welker showed that a Cohen-Macaulay affine semigroup ring of minimal multiplicity is nonstandard Koszul \cite[Theorem~5.2]{HRW98}. Thus, minimal multiplicity gives a useful sufficient condition for $S^{(e)}$ to be nonstandard Koszul in every dimension. For $n\geq 3$ there are gradings for which $S^{(e)}$ does not have minimal multiplicity. 
	\item A proof of Backelin's theorem uses that truncations of sufficiently high degree of a standard $\ZZ$-graded $\KK$-algebra have linear resolutions. Davis and Martinova prove that $S_{\geq d}$ is a nonstandard Koszul module, which constrains the resolution of $S_{\geq d}$ \cite[Theorem~A]{DM25}. 
\end{enumerate}

Theorem~\ref{thm:main-ex} shows that these positive results do not extend from sufficiently divisible Veronese parameters to all sufficiently large
ones. Its proof is based on finding an explicit minimal cubic generator in the presentation of $\gr_{\mfm}(S^{(e)})$ through an elementary combinatorial construction. For a weight vector $\ww \in \ZZ_{\geq1}^{n}$ let 
\[
\Lambda_e(\ww)\coloneqq \left\{\vv\in\ZZ_{\geq0}^n\;\; \big| \;\; \ww\cdot\vv=e\right\}.
 \]
For $\cc \in \Lambda_{e}(\ww)$, a \emph{nontrivial barycenter decomposition} of $\cc$ is an equality $3\cc=\uu+\vv+\zz$ where $\uu,\vv,\zz \in \Lambda_{e}(\ww)$ and $(\uu,\vv,\zz)\neq(\cc,\cc,\cc)$. We say that the point $\cc$ is \emph{midpoint-rigid} if $\pp+\qq=2\cc$ for $\pp,\qq\in \Lambda_{e}(\ww)$ implies that $\pp=\qq=\cc$. We call a tuple $(\cc;\uu,\vv,\zz)$ consisting of a midpoint-rigid point $\cc \in \Lambda_{e}(\ww)$ and a nontrivial barycenter decomposition $3\cc=\uu+\vv+\zz$ a \emph{cubic obstruction configuration}. Under the bijection between degree-$e$ monomials of $S$ and points of $\Lambda_e(\ww)$, a cubic obstruction configuration corresponds to monomials $x^{\cc}, x^{\uu}, x^{\vv}, x^{\zz} \in S$ such that $(x^{\cc})^{3}=x^{\uu}x^{\vv}x^{\zz}$. This relation gives a cubic equation in the tangent cone of $S^{(e)}$. The midpoint-rigidity of $\cc$ forces this cubic to be a minimal generator, implying $\gr_{\mfm}(S^{(e)})$ is not Koszul. 
\begin{example}
	Consider the setup of Theorem~\ref{thm:main-ex}, where $S=\KK[x_1,x_2,x_3,x_4]$ with $\deg(x_1)=1, \deg(x_2)=4, \deg(x_3)=7, \deg(x_4)=9$ and $e=21$. The $30$ indecomposable monomials in $\mfm \setminus \mfm^{2}$, including every monomial of degree 21, are generators for  $\gr_{\mfm} (S^{(21)})$. Writing $\Gamma$ for this set gives the canonical presentation $\pi:\KK[T_{h} \; | \; h \in \Gamma] \to \gr_{\mfm}(S^{(21)})$. The vectors $\cc=(1,1,1,1)$, $\uu=(3,0,0,2)$, $\vv=(0,3,0,1)$, $\zz=(0,0,3,0)$ form a cubic obstruction configuration in $\Lambda_{21}(\ww)$. The corresponding monomials in $S_{21}$ are $f = x_1x_2x_3x_4$, $g_{1} = x_{1}^{3}x_{4}^2$, $g_{2} = x_{2}^{3}x_{4}$, and $g_{3} = x_3^{3}$. The relation $f^{3}=g_{1}g_{2}g_{3}$ gives a cubic $C=T_f^3-T_{g_1}T_{g_2}T_{g_3} \in \ker(\pi)$, which is a minimal generator, implying $\gr_{\mfm}(S^{(21)})$ is not Koszul.  In fact $\ker(\pi)$ has $335$ minimal generators, of which this is the unique nonquadratic one.
\end{example}

 This construction readily generalizes to more than four variables. 

\begin{corollaryalpha}\label{cor:more-vars}
	For every $n\geq4$ there exists a positively $\ZZ$-graded polynomial ring with $n$-variables  for which Conjecture~\ref{con:DEM} fails. 
\end{corollaryalpha}

One might hope that the counterexamples given in Theorem~\ref{thm:main-ex} are somewhat sparse, and perhaps Conjecture~\ref{con:DEM} holds generically in some sense. We show that this is unlikely as we extend the method that led to the counterexample above into a positive-density family of counterexamples.

\begin{theoremalpha}\label{thm:density}
	Let $\cB \subset \ZZ^4_{\geq1}$ be the set of primitive degrees $\ww \in \ZZ^{4}_{\geq1}$ such that Conjecture~\ref{con:DEM} fails for $S=\KK[x_1,x_2,x_3,x_4]$ where $\deg(x_i)=w_i$. Then 
	\[
	\liminf_{N\to \infty} \frac{\#\left(\cB \cap [1,N]^{4}\right)}{N^{4}} \geq \frac{3}{40\pi^2}\approx 0.007599.
	\]
\end{theoremalpha}

 These examples show that the usual positivity heuristic and nonstandard Koszul-ness are more subtle in the multigraded setting. In particular, this seems to highlight how for nonstandard gradings  the concepts of ``sufficiently large'' and ``sufficiently divisible'' diverge. Sufficiently divisible Veronese subrings are eventually nonstandard Koszul, but sufficiently large ones need not be.  

 The only remaining case of Conjecture~\ref{con:DEM} is $n=3$. Proposition~\ref{prop:three-vars} shows that, for any fixed three-variable grading, cubic obstruction configurations occur in only finitely many degrees; thus our method cannot decide Conjecture~\ref{con:DEM} in this case. Further, perhaps Conjecture~\ref{con:DEM} can be corrected by requiring the weights be ``nice'' in some sense. Recent work of Banks and Ramkumar on varieties of minimal degree in weighted projective space suggests one possible definition of ``nice'' might be that $w_1 =1$ and $w_{i} |w_{i+1}$, i.e. what they call divisible weight \cite[Definition~1.1]{mayaRitvik26}. Interestingly, cubic obstruction configurations cannot occur when the weights are divisible (see Remark~\ref{rem:divisible-weights}). It is unclear what to expect in both of these cases, and we would not wager either way. 
 
 \subsection*{Acknowledgements} I am grateful to Maya Banks, Caitlin Davis, Daniel Erman, Boyana Martinova, and Alexandra Seceleanu for generous conversations from which this paper grew and for comments on an earlier draft. Computations in \textit{Macaulay2} \cite{M2} were essential to this article. ChatGPT 5.6 Sol and Claude Opus 5 were used for proofreading.

\section{The Combinatorial Construction}

Fix $n \geq 1$ and let $S=\KK[x_1,\ldots,x_n]$ be the positively $\ZZ$-graded polynomial ring where $\deg(x_i) = w_i \in \ZZ_{\geq1}$. Fix an integer $e\geq 1$, and let $\mfm \subset S^{(e)}$ denote the homogeneous maximal ideal. Let $\Gamma$ denote the set of monomials in $\mfm \setminus \mfm^2$. Thus, $\Gamma$ is the set of indecomposable nonconstant monomials of $S^{(e)}$, and its residue classes form a $\KK$-basis of $\mfm/\mfm^2$. Consequently, if $R = \Sym^{\bullet}\left(\mfm/\mfm^{2}\right)$, which is a standard $\ZZ$-graded polynomial ring, then  $R\cong \KK[T_{f} | f \in \Gamma]$. There is a canonical graded presentation
\[
\begin{tikzcd}[row sep = .25em, column sep = 3.75em]
R \arrow[r, two heads, "\pi"] &  \gr_{\mfm}\left(S^{(e)}\right)  & \text{given by} &
T_{f} \arrow[r, mapsto] & f+\mfm^2
\end{tikzcd},
\]
and so $\gr_{\mfm}(S^{(e)}) \cong R/J$, where $J=\ker(\pi)$. Since $\pi$ is graded, we may break $\pi$ into its graded pieces $\pi_{d}:R_{d} \to \mfm^{d}/\mfm^{d+1}$ and $J_d = \ker(\pi_d)$ for all integers $d\geq0$. As $\mfm^{d}$ and $\mfm^{d+1}$ are monomial ideals, the classes of the monomials in $\mfm^{d}\setminus \mfm^{d+1}$ form a $\KK$-basis for $\mfm^{d}/\mfm^{d+1}$. 
\begin{lemma}[Cubic Obstruction]\label{lem:obstruction}
	Let $f, g_{1},g_{2},g_{3} \in S$ be monomials of degree $e$ such that $f^{3}=g_1g_2g_3$ and $T_{f}^{3} \neq T_{g_1}T_{g_2}T_{g_3}$. If every degree-$e$ monomial divisor of $f^2$ is equal to $f$ then $T_{f}^{3}-T_{g_1}T_{g_2}T_{g_3}$ is a minimal cubic generator of $J$. Therefore, $S^{(e)}$ is not nonstandard Koszul. 
\end{lemma}

\begin{proof}
	Every monomial of $S$ of degree $e$ is indecomposable in $S^{(e)}$. Thus, the variables $T_{f}, T_{g_1}, T_{g_2}$, and $T_{g_3}$ all belong to $R$. The equality $f^{3}=g_1g_2g_3$ implies that $C\coloneqq T_{f}^{3}-T_{g_1}T_{g_2}T_{g_3} \in J$. The second assumption ensures $C\neq 0$ in $R$. 
	We first claim that no element of $J_2$ contains $T_f^2$ with nonzero coefficient.  Indeed, $f^2\in\mfm^2\setminus\mfm^3$, since every
monomial in $\mfm^3$ has weighted degree at least $3e$.  Thus, if a quadric in $J$ contained $T_f^2$, cancellation in $\mfm^2/\mfm^3$ would require another monomial $T_uT_v$ with $uv=f^2$. Because $u,v\in\Gamma\subset\mfm$ and $\deg(f^2)=2e$, both $u$ and $v$ have degree $e$.  The hypothesis then gives $u=v=f$, a contradiction.
	
Since $J_1=0$, the cubic $C$ fails to be a minimal generator if and only if $C \in R_{1}J_{2}$. Towards a contradiction assume $C$ were not a minimal generator and so $C =\sum_{i}  L_iQ_i$
 where $L_{i} \in R_1$ and $Q_{i} \in J_{2}$.  Since $T_{f}^{3}$ appears on the left-hand side, there must be an $i$ such that $T_{f}^{3}$ appears as a nonzero term in $L_iQ_i$. In particular, $Q_i\in J_2$ must contain $T_f^2$ with nonzero coefficient, which contradicts the preceding paragraph. Therefore $C$ is a minimal cubic generator of $J$. Finally, every Koszul algebra is quadratic, so $\gr_{\mfm}(S^{(e)})$ is not Koszul. Hence $S^{(e)}$ is not nonstandard Koszul. 
\end{proof}

\begin{remark}\label{rem:comb-reform}
	Under the bijection between degree-$e$ monomials and points of $\Lambda_e(\ww)$, the hypotheses of Lemma~\ref{lem:obstruction} say exactly that the corresponding exponent vectors $(\cc;\uu,\vv,\zz)$ form a cubic obstruction configuration. Indeed, $f^3=g_1g_2g_3$ is equivalent to $3\cc=\uu+\vv+\zz$, while a degree-$e$ monomial divisor of $f^2$ with exponent vector $\pp$ corresponds to a pair $\pp,2\cc-\pp\in\Lambda_e(\ww)$.
\end{remark}

\begin{proof}[Proof of Theorem~\ref{thm:main-ex}]
	Let $\ww=(1,4,7,9) \in \ZZ^4$. Let $e = 9k + 12$ for an integer $k\geq1$. Consider the following four monomials of degree $e$ in $S$:
	\[
	f = x_1x_2x_3x_4^{k}, \quad\quad g_{1} = x_{1}^{3}x_{4}^{k+1}, \quad\quad g_{2} = x_{2}^{3}x_{4}^{k}, \quad\quad g_{3} = x_3^{3}x_4^{k-1}.
	\]
	We now wish to apply Lemma~\ref{lem:obstruction} to these monomials. The identity $f^3 = g_{1}g_{2}g_{3}$ is immediate. Since these are nonzero distinct indecomposable monomials the corresponding cubic monomials in $R$ are distinct. It remains to check that the factorization of $f^2$ into degree $e$ monomials is the obvious one. Using vector notation, if $\cc=(1,1,1,k)$ is the exponent vector of $f$, such a factorization corresponds to  $2\cc=\uu+\vv$ for $\uu,\vv \in \Lambda_{e}(\ww)$. Write $\uu=\cc+\delta$ and $\vv=\cc-\delta$ where $\delta = (\delta_1,\delta_2,\delta_3,\delta_4)\in \ZZ^4$. Since $\cc=(1,1,1,k)$ the non-negativity of the coordinates of $\uu$ and $\vv$ imply $\delta_1,\delta_2,\delta_3 \in \{-1,0,1\}$ and $\cc,\uu,\vv$ having the same weighted degree forces 
	\begin{equation}\label{eq:main}
	\delta_1+4\delta_2+7\delta_3+9\delta_4 = 0.
	\end{equation}
	We must show $\delta_i=0$ for all $i$. Reducing this equation modulo 3 shows that $\delta_1+\delta_2+\delta_3$ must be divisible by $3$, meaning $\delta_1+\delta_2+\delta_3$ is either $-3$, $0$, or $3$. If $\delta_1+\delta_2+\delta_3$ is $3$ then $\delta_1=\delta_2=\delta_3=1$ and equation~\eqref{eq:main} becomes $12+9\delta_4=0$, which has no integer solutions. The same argument rules out the sum $\delta_1+\delta_2+\delta_3$ being $-3$, and so $\delta_1+\delta_2+\delta_3=0$. Substituting $\delta_1=-\delta_2-\delta_3$ into equation~\eqref{eq:main} and dividing each side by 3 simplifies to $\delta_2+2\delta_3+3\delta_4=0$. This means that $\delta_2 \equiv \delta_3 \pmod{3}$, however, as both are in $\{-1,0,1\}$ we have $\delta_2=\delta_3$. Combining this with $\delta_1+\delta_2+\delta_3=0$ implies that $\delta_1=\delta_2=\delta_3=0$ from which we also get that $\delta_4=0$. Thus $\delta=0$, and so $f$ is the only degree-$e$ monomial divisor of $f^2$. Lemma~\ref{lem:obstruction} now shows that $S^{(e)}$ is not
nonstandard Koszul. Since $e=9k+12$ with $k\geq1$ is equivalent to $e\geq21$ and $e\equiv3\pmod{9}$, this proves the theorem.
\end{proof}

\begin{proof}[Proof of Corollary~\ref{cor:more-vars}]
	For $n\geq4$ let $S=\KK[x_1,\ldots,x_n]$ where $\deg(x_1)=1$, $\deg(x_2)=4$, $\deg(x_3)=7$, $\deg(x_4)=9$, and $\deg(x_i)$ is any positive integer for $i\geq 5$. Going through the argument in the proof of Theorem~\ref{thm:main-ex}, one sees that the four monomials used there satisfy the conditions for Lemma~\ref{lem:obstruction} in this larger ring and hence give a counterexample here. 
\end{proof}

The next proposition explains the limitation of this construction in three variables. Midpoint rigidity forces two coordinates of $\cc$ to be small, while a nontrivial barycenter decomposition forces $\uu-\cc$, $\vv-\cc$, and $\zz-\cc$  to be larger in three distinct coordinate directions. Together these observations bound the degree in terms of the weights.

\begin{proposition}\label{prop:three-vars}
	Let $n=3$. Fix $\ww \in \ZZ_{>0}^3$ and let $W = \max\{\ww_1, \ww_2, \ww_3\}$. If $\Lambda_{e}(\ww)$ contains a cubic obstruction configuration then $1\leq e < 4W^{2}$. 
\end{proposition}

\begin{proof}
	Let $S=\KK[x_1,x_2,x_3]$ with $\deg(x_i) = \ww_i$ for $\ww\in \ZZ^3_{\geq1}$.  Suppose we have a cubic obstruction configuration in degree $e\geq1$, i.e., vectors $\cc,\uu,\vv,\zz \in \Lambda_{e}(\ww)$ where $\cc$ is midpoint-rigid and $3\cc=\uu+\vv+\zz$ is a nontrivial barycenter decomposition. We will show that we can bound $e$ in terms of $\ww$. Set
	\[
	L(\ww) \coloneqq \{\delta \in \ZZ^3 \; | \; \ww\cdot \delta =  0\} \quad \quad \text{and} \quad \quad M(\ww,\cc) \coloneqq L(\ww) \cap \prod_{i=1}^{3} \left[-\cc_i, \cc_i\right].
	\]
Midpoint rigidity of $\cc$ is equivalent to $M(\ww,\cc) = \{\zero\}$. If $\cc_i, \cc_j\geq W$ for $i\neq j$, then the vector whose only nonzero entries are $\ww_j$ in the $i$-th coordinate and $-\ww_i$ in the $j$-th coordinate is contained in $M(\ww,\cc)$. Therefore, at most one coordinate of $\cc$ is greater than or equal to $W$.  
	
	Consider the difference vectors $\delta_{\uu}=\uu-\cc$, $\delta_{\vv}=\vv-\cc$, and $\delta_{\zz}=\zz-\cc$. Midpoint rigidity together with the nontriviality of the barycenter decomposition forces each of these differences to be nonzero. For example, if $\delta_{\uu}=0$ then $\uu=\cc$ and so $3\cc=\uu+\vv+\zz$ simplifies to $2\cc=\vv+\zz$ and midpoint rigidity implies $\vv=\zz=\cc$ contradicting the nontriviality of the barycenter decomposition. Further, each difference vector is in $L(\ww)$ and is coordinate-wise greater than or equal to $-\cc$. Define $E_{\uu} = \{i \; | \; (\delta_{\uu})_i>\cc_i\}$, $E_{\vv} = \{i \; | \; (\delta_{\vv})_i>\cc_i\}$, and $E_{\zz} = \{i \; | \; (\delta_{\zz})_i>\cc_i\}$. If any of these sets were empty, then the corresponding nonzero difference vector would lie in $M(\ww,\cc)$, contradicting midpoint rigidity. Moreover, these sets are pairwise disjoint. For example, if $i\in E_{\uu}\cap E_{\vv}$, then $(\delta_{\zz})_i =-(\delta_{\uu})_i-(\delta_{\vv})_i <-2\cc_i\leq-\cc_i$, contradicting $\zz_i\geq0$. Since $E_{\uu},E_{\vv},E_{\zz}$ are nonempty pairwise disjoint subsets of $\{1,2,3\}$, each is a singleton and together they exhaust all three coordinates. Choose $k$ so that $\cc_i<W$ for every $i\neq k$; this is possible because at most one coordinate of $\cc$ is at least $W$. After relabeling $\uu,\vv,\zz$, we may assume that $k\in E_{\uu}$. Then $\uu_k>2\cc_k$, and so $2\ww_k\cc_k<\ww_k\uu_k\leq\ww\cdot\uu=e$.  Since $\ww\cdot\cc=e$, it follows $e<2\sum_{i\neq k}\ww_i\cc_i<4W^2$. 
\end{proof}

\begin{remark}
	The constructions in \cite[Proposition~4.3]{CFHMNSS26} are examples of cubic obstruction configurations in the sense above. However, their examples are supported on only three variables. Proposition~\ref{prop:three-vars} shows that their configurations cannot directly be used to disprove Conjecture~\ref{con:DEM}.
\end{remark}

\begin{remark}\label{rem:divisible-weights}
	If the weights $\ww=(w_1,\ldots,w_n) \in \ZZ^{n}_{\geq1}$ are divisible in the sense of \cite[Definition~1.1]{mayaRitvik26}, i.e. after reordering $w_1=1$ and $w_{1} | w_{2} |\cdots | w_{n}$, then no cubic obstruction configurations occur. The divisibility of $\ww$ implies that every monomial supported on $x_1,\ldots,x_{n-1}$ of degree at least $w_{n}$ is divisible by a monomial of degree exactly $w_n$. This constrains midpoint rigidity; if  $\cc \in \Lambda_{e}(\ww)$ and $\cc$ is midpoint rigid then $c_{n}=0$ or $c_{n} = \lfloor \frac{e}{w_n} \rfloor$. This extremity forces any barycenter decomposition of $\cc$ to be trivial.
\end{remark}

\section{Density of Counterexamples}

Fix a positive integer $m$. For $z\in \ZZ$, write $\overline{z} \in \{0,1,\ldots,m-1\}$ for its nonnegative residue modulo $m$. For $\alpha \in \RR$, set $\|\alpha\|_{m} \coloneqq \min\{ |\alpha-km| \; : \: k \in \ZZ\}$, which is the distance from $\alpha$ to the nearest integer multiple of $m$.  For $(\alpha,\beta) \in \RR^2$ set
\[
\Delta_{m}(\alpha,\beta) \coloneqq \min\left\{ \|\alpha\|_{m}, \;\; \|\beta\|_{m}, \;\; \|\alpha+\beta\|_{m}, \;\; \|\alpha-\beta\|_{m} \right\}.
\]
The proof of Theorem~\ref{thm:density} amounts to a careful construction of cubic obstructions. The construction has two requirements. Certain linear inequalities ensure that the exponent vectors in the barycenter decomposition are nonnegative, while a modular separation condition rules out nontrivial midpoint decompositions. We package these requirements into the following open set: 
\[
P_{m} \coloneqq \left\{ (\alpha,\beta) \in (0,m)^{2} \;\; \bigg| \;\;
\begin{matrix}
0<2\alpha-\beta, \quad 
0<2\beta-\alpha-m, \\ 
\max\{2\alpha-\beta, 2\beta-\alpha-m\} < \Delta_m(\alpha,\beta)
\end{matrix}
\right\}.
\]
One checks that $P_{m}$ is the interior of the quadrilateral with vertices $(m/3, 2m/3)$, $(m/2, 3m/4)$, $(m/2, 5m/6)$, and $(2m/5, 4m/5)$. In particular, $P_{m}=mP_{1}$ and $\vol(P_{m})=m^2/80$. For each $m\geq2$, the construction below will produce an explicit injection 
\begin{equation}\label{eq:inj}
\left\{1\leq q < m \;\; \big| \;\; (q,m) = 1\right\} \times \left(P_m \cap \ZZ^2 \right) \hookrightarrow \cB \cap \left([1,m-1]^3 \times \{m\}\right)
\end{equation}
where $(q,(a,b))\mapsto (q, \overline{qa}, \overline{qb}, m)$. The content of this construction is to show that for the polynomial ring $S = \KK[x_1,x_2,x_3,x_4]$ with degree vector $ (q, \overline{qa}, \overline{qb}, m)$ we can find infinitely many values of $e$ for which we can invoke Lemma~\ref{lem:obstruction}. 

\begin{proof}[Proof of Theorem~\ref{thm:density}]
	Fix $m\geq 2$, choose an integer $1 \leq q < m$ that is a unit modulo $m$, and let $(a,b) \in P_m\cap\ZZ^2$. Set $\ww=\ww(q,a,b,m)\coloneqq (q, r_1,r_2,m)$, where $r_1 =\overline{qa}$ and $r_{2}=\overline{qb}$. Since $0<a,b<m$, the residues $r_1$ and $r_2$ are nonzero. Thus, $\ww \in \ZZ^{4}_{\geq1}$ and $\ww$ is primitive since $(q,m)=1$. Put $A = 2a-b$, $B=2b-a-m$, and $M = \max\{A,B\}$. 
	Since $(a,b) \in P_m$, by definition $A$ and $B$ are positive and 
	\begin{equation}\label{eq:min-m-dist} 
	0< M = \max\{2a-b, 2b-a-m\} < \Delta_{m}(a,b) =  \min\left\{ \|a\|_{m}, \;\; \|b\|_{m}, \;\; \|a+b\|_{m}, \;\; \|a-b\|_{m} \right\} \leq \frac{m}{2}.
\end{equation}
By the definition of $r_1$ and $r_2$ the integers 
	\[
	\rho \coloneqq \frac{Aq-2r_1+r_2}{m} \quad \quad \text{and} \quad \quad \sigma \coloneqq \frac{Bq + r_1 - 2r_2}{m}
	\]
	are well-defined integers. For every integer $t\geq  |\rho|+|\sigma|$ we will consider the integer vectors
	\[
	\cc \coloneqq (M,1,1,t), \;\; \uu \coloneqq (M-A, 3, 0, t+\rho), \;\; \vv = (M-B,0,3,t+\sigma), \;\; \text{and} \;\; \zz = (M+A+B, 0,0, t-\rho-\sigma),
	\]
	whose dependence on $t$ we suppress for ease of notation. If  $e_t \coloneqq \ww \cdot \cc= qM+r_1+r_2+tm$, then a direct computation shows that $\ww \cdot \uu = \ww \cdot \vv = \ww \cdot \zz = e_t$. Further, these vectors give a barycenter decomposition, $3\cc = \uu + \vv + \zz$. Since $\cc \neq \uu$ this is a nontrivial decomposition. 
	
	We claim that $\cc$ is not a midpoint of two distinct points of $\Lambda_{e_t}(\ww)$. Suppose that $\cc-\delta, \cc+\delta \in \Lambda_{e_t}(\ww)$ for some $ \delta \in \ZZ^4$. We will show that $\delta$ must be the zero vector. Nonnegativity of the coordinates of $\cc\pm\delta$ gives us that $|\delta_1| \leq M$ and $\delta_2,\delta_3 \in \{-1,0,1\}$. Equality of weighted degrees implies 
	\begin{equation}\label{eq-weight-deg}
	q\delta_1+ r_{1}\delta_2 + r_2\delta_3 + m\delta_4 = 0.
\end{equation}
Reducing \eqref{eq-weight-deg} modulo $m$ and using $r_1\equiv qa \pmod{m}$ and $r_2\equiv qb\pmod{m}$, we obtain
\[
q\left(\delta_1+a\delta_2+b\delta_3\right) \equiv 0 \pmod{m}.
\]
Since $q$ is a unit modulo $m$, it follows that
\begin{equation}\label{eq:delta-arg}
	\delta_1 + a \delta_2 + b\delta_3 \equiv 0 \pmod{m}.
\end{equation}
	If $(\delta_2,\delta_3) \neq (0,0)$ then, up to sign, $a \delta_2 + b\delta_3$ is one of $a$, $b$, $a+b$, or $a-b$. Consequently, combining \eqref{eq:min-m-dist} and \eqref{eq:delta-arg} gives 
	\[
	M < \|a\delta_2+b\delta_3\|_{m} = \|\delta_1\|_m = |\delta_1| \leq M,
	\]
	which is a contradiction. Thus, $\delta_2=\delta_3=0$. Equation~\eqref{eq:delta-arg} together with $|\delta_1|\leq M < m/2$ implies $\delta_1=0$, and from all of this \eqref{eq-weight-deg} gives $\delta_4=0$. 
	Therefore, Lemma~\ref{lem:obstruction} and Remark~\ref{rem:comb-reform}  give a cubic obstruction to $S^{(e_t)}$ being nonstandard Koszul for every $t\gg0$. Since $e_t$ tends to infinity as $t\to \infty$, the primitive weight vector $\ww$ belongs to $\cB$. 
	
	It remains to count the degrees produced this way. For $m\geq2$ the above construction gives a map displayed in \eqref{eq:inj} given by $(q,a,b,m)\mapsto (q, \overline{qa}, \overline{qb}, m)$. This is an injection; the first coordinate recovers $q$ and last recovers $m$. Once $q$ and $m$ are known, multiplication by $q^{-1}$ modulo $m$ recovers $a$ and $b$ from $r_1$ and $r_2$; because $0<a,b<m$, these residues determine $a$ and $b$ uniquely. 

To count the source of \eqref{eq:inj}, let $\overline{P}_1$ denote the closure of $P_1$. By Ehrhart's theorem for rational polygons, $\#(m\overline{P}_1\cap\ZZ^2) = \vol(P_1)m^2+O(m)$  \cite[Theorem~2.11]{beckRobins15}. The boundary $\partial(m\overline{P}_1)$ contains $O(m)$ lattice points, so the same estimate holds for the interior $\#(P_m\cap\ZZ^2)=\frac{m^2}{80}+O(m)$.
Summing over $m\leq N$ shows that 
\begin{align*}
\#\left(\cB \cap [1,N]^4\right) & \geq \sum_{m \leq N} \phi(m) \cdot \#\left(P_m \cap \ZZ^2\right) 
= \frac{1}{80}\sum_{m\leq N} m^2\phi(m) - O\left(\sum_{m\leq N} m\phi(m)\right).
\end{align*}
Finally, standard estimates,
 see \cite[Theorem~3.7]{apostol76}, show that 
\begin{equation}\label{eq-bound}
\#\left(\cB\cap [1,N]^{4}\right) \geq \frac{3}{160 \pi^2}N^4-O\left(N^3\log(N)\right).
\end{equation}
Every degree vector produced by this construction has a unique largest coordinate, namely its fourth coordinate $m$. Since $\cB$ is invariant under permutations of the coordinates, moving the fourth coordinate to each of the four possible positions produces four pairwise disjoint families of the same cardinality. Multiplying \eqref{eq-bound} by 4, dividing by $N^4$ and taking $\liminf$ proves the claim. 
\end{proof}

\begin{remark}
	When $m=9$, $q=1$ and $(a,b) = (4,7)$ the construction recovers the configuration used  in the proof of Theorem~\ref{thm:main-ex}, up to permuting $\uu$, $\vv$, $\zz$. With these choices one gets that $\ww=(1,4,7,9)$, $A=B=M=1$, $\rho=0$, and $\sigma=-1$, giving $\cc=(1,1,1,t)$, $\uu=(0,3,0,t)$, $\vv=(0,0,3,t-1)$, and $\zz=(3,0,0,t+1)$. 
	
\end{remark}

\bibliographystyle{alphaurl}
\bibliography{references.bib}

\end{document}